\documentclass[12pt]{amsart}
\usepackage{amsmath,amsxtra,amssymb,latexsym, amscd,amsthm,amsfonts}
\newtheorem{thm}{Theorem}[section]
\newtheorem{cor}[thm]{Corollary}
\newtheorem{lem}[thm]{Lemma}
\newtheorem{prop}[thm]{Proposition}
\theoremstyle{definition}
\newtheorem{defn}[thm]{Definition}
\theoremstyle{remark}
\newtheorem{rem}[thm]{Remark}
\numberwithin{equation}{section}
\newcommand{\R}{\mathbb R}
\newcommand{\N}{\mathbb N}

\newcommand{\simuleq}[1]{\left\{\begin{aligned}#1\end{aligned}\right.}%

\begin{document}%

\title[]{  Some regularity properties of viscosity solution defined by Hopf formula}%
\author{NGUYEN HOANG}%
\address{Department of Mathematics,
College of Education, Hue University, 32 LeLoi, Hue, Vietnam}%
\email{nguyenhoanghue@gmail.com, nguyenhoang@hueuni.edu.vn}%


\begin{abstract}
Some properties of characteristic curves in connection with viscosity solution of Hamilton-Jacobi equations $(H,\sigma)$ defined by Hopf formula $u(t,x)=\max_{q\in\R^n}\{
\langle x,q\rangle -\sigma^*(q)-tH(q)\}$ are
studied. We are concerned with the points where the solution  $u(t,x)$ is differentiable, and the strip of the form
$\mathcal R=(0,t_0)\times \R^n$ of the domain $\Omega$ where  $u(t,x)$ is of class $C^1(\mathcal R).$ Moreover, we investigate the propagation of singularities in forward of this solution.
\end{abstract}
\maketitle

\section{Introduction}
Consider the Cauchy problem for Hamilton-Jacobi equation $(H,\sigma):$
 \begin{equation}\label{1.1}u_t + H(t,x,D_x u)=0\, , \,\, (t,x)\in \Omega=(0,T)\times \R^n,
\end{equation}
\begin{equation}\label{1.2}
u(0,x)=\sigma(x)\, , \,\, x\in \R^n.\end{equation}

It is well-known that, due to the nonlinearity of the Hamiltonian $H(t,x,p)$  in general, the smooth solutions of the problem exist in a narrow neighborhood  of the hyperplane $t=0.$ The studies of global solutions (i.e. the solutions defined on whole domain $\Omega$) of the Cauchy problems started in the decade of 1950s with some notions of generalized solutions. At the beginning, the notion of Lipschitz solution was taken into account. Typically, the solution is defined as a locally Lipschitz function $u(t,x)$ satisfying the equation (1.1) almost everywhere on $\Omega$ and $u(0,x)=\sigma(x),\ x\in \R^n.$ Unfortunately, this type of generalized solutions of the problem is not unique, thus one must restrict to consider the solutions in some specific class of functions.

In 1965,  Hopf \cite{h} proved two well-known formulas for representation of Lipschitz solutions of the Hamilton - Jacobi equations $(H,\sigma):$

\begin{equation}\label{1.3}u(t,x)=\min_{y\in \R^n} \Big\{ \sigma
(y)+tH^*\big (\frac {x-y}{t}\big)\Big \}, \end{equation} where
$H=H(p)$ is convex and superlinear, $\sigma$ is Lipschitz on $\R^n,$  and

\begin{equation}\label{1.4}u(t,x)=\max_{q\in\R^n}\{
\langle x,q\rangle -\sigma^*(q)-tH(q)\}\end{equation} if the function $H=H(p)$ is continuous;  $\sigma (x)$
is a convex and Lipschitz function, see \cite{h,be,lr}. Here * denotes the Fenchel conjugate.

These formulas are called Hopf formulas. Note that, when $n=1,$ the formula  (1.3) was proved by Lax \cite{la} in 1958, therefore this formula is also called Hopf - Lax formula.

In 1983, Crandall M.G. and
Lions P.L. in \cite{cl} first introduced the notion of viscosity solution that plays 
a fundamental role in studying Hamilton-Jacobi equations as well as the related problems such as
calculus of variation, optimal control theory, differential games, etc. By definition, a
viscosity solution of Hamilton-Jacobi equation is merely a
continuous function $u$ satisfying differential inequalities associated with the equations. 

In 1984, Bardi - Evans \cite{be} proved that the function $u(t,x)$ defined by (1.3) as well as  (1.4) is a viscosity solution of the corresponding problem $(H,\sigma).$

It is noticed that, the representation formula (1.3) and its generalization for the case $H=H(t,x,p)$ were widely studied under an essential assumption that $H(t,x,p)$ is a convex function in $p$ (see \cite{bado, cs} and references therein). Actually, ones then proved that the value function of a calculus of variation problem or an optimal control problem is a viscosity solution of the asscociated dynamic programming equation where the Hamiltonian is convex in the gradient variable. Many important results on the theoretical aspect as well as applications were obtained. Especially, regularity properties as well as propagation of singularities of viscosity solutions in the case of convex Hamiltonians are intensively studied, see \cite{ac,ac1,bcjs,fl}
especially \cite{cs} and references therein.

On the other hand, in the theory of differential games,  Hamiltonians of the associated dynamic programming equations are neither convex nor concave in general, see \cite{ca,e2}).  Nevertheless, not so many studies are realized for the case,  even  for simple Hamiltonian $H=H(p).$ In \cite{bafa},  Bardi M. and Faggian  S. presented explicit estimates below and above of the form ``maxmin" and ``minmax'' for the viscosity solutions  where either the Hamiltonians or the initial data are not necessarily convex, but can be expressed as the sum of a convex and a concave function. Recently, Evans \cite{e1, e2} establishes a general representation formula for nonconvex Cauchy problem $(H, \sigma)$ by methods of  ``generalized envelopes'' and ``adjoint and compensated compactness".

This paper is devoted to studying some regularity properties of viscosity solution $u(t,x)$ given by Hopf formula (1.4). We investigate the differentiability of $u(t,x)$ on the characteristic curves and define some strips of the form $\mathcal R=(0,t_0)\times \R^n\subset\Omega$ so that $u(t,x)\in C^1(\mathcal R)$ basing on the set of maximizers $\ell(t,x)$ in the formula (1.4). This can be considered as a bridge to fill the gap between viscosity solution and classical solution for nonconvex Hamiltonian.

The structure of the paper is as follows. In section 2 we present some necessary notions and properties of Hopf formula and viscosity solution. Next, in section 3, we suggest a classification  of characteristic curves  at each
 point of the domain and then study the differentiability properties of Hopf formula $u(t,x)$ on these curves.
In section 4, we establish various conditions based on the characteristics so that $u(t,x)$ defined by (1.4)
 is continuously differentiable on the strip of the form $(0,t_0)\times \R^n.$
 In the last section, we analyse the formation of singularities of the Hopf formula  $u(t,x)$ and show that the singularities may propagate forward from $t$-time $t_0$ to the boundary of the domain. Illustrative examples are  given.

This paper can be considered as a continuation of \cite{nh1} to the
cases where dimension of state variable $n$ is greater than 1. To the best of our knowledge, the results obtained here are new and significant in understanding the regularity of viscosity solution of the problem $(H,\sigma)$ with nonconvex Hamiltonian.

\medskip

 We use the following notations. Let  $T$ be a positive number, $\Omega =(0,T)\times
\R^n;\, |\, .\, |$ and $\langle .,.\rangle $ be the Euclidean norm and
the scalar product in $\R^n$, respectively, and let $B'(x_0,r)$ be the closed ball centered at $x_0$ with radius $r.$ For a function $u(t,x)$ defined on $\Omega,$ we denote $ D_xu=(u_{x_1},\dots, u_{x_n})$ and $Du=(u_t, D_xu).$ 

\section{Hopf formula and viscosity solution}

We now  consider the  Cauchy problem for Hamilton-Jacobi equation:

 \begin{equation}\label{2.1}u_t + H(D_x u)=0\, , \,\,
(t,x)\in \Omega =(0,T)\times \R^n,\end{equation}
\begin{equation}\label{2.2}u(0,x)=\sigma(x)\, ,
\,\, x\in \R^n,\end{equation} where the Hamiltonian $H(p)$  is a continuous function and $\sigma (x)$ is a convex function on $\R^n.$

\medskip

We assume a  compatible condition for $H(p)$ and $\sigma
(x)$ as follows:

\smallskip

(Hf1) :  {\it For every $(t_0,x_0)\in
[0,T)\times\R^n $,  there exist positive constants $r$ and $N$
such that
$$\langle x,p\rangle -\, \sigma ^* (p)-tH(p) <
 \max_{|q|\le N}\{\langle x,q\rangle -\, \sigma ^* (q)-tH(q)\},$$
whenever $ (t,x)\in [0,T)\times\R^n,\, |t-t_0|+|x-x_0|<r$ and
$| p| >N.$}

\medskip

Let $\sigma^*$ be the Fenchel conjugate of $\sigma.$ We denote by 
$$D=  {\rm dom}\, \sigma^*=\{y\in \R^n\ |\, \sigma^*(y)<+\infty\}$$ 
the effective domain of the convex function $\sigma^*.$

\medskip
From now on, the Hopf formula for Problem (2.1) - (2.2) is a function defined by
\begin{equation}\label{2.3}u(t,x)=\max_{q\in \R^n}\, \{\langle x,q\rangle -\, \sigma^*
(q)-t H(q)\}.\end{equation}
Let
\begin {equation} \label {2.4}\varphi (t,x,q) =\langle
x,q\rangle -\, \sigma ^* (q)-t H(q),\ (t,x)\in \Omega,\ q\in \R^n .\end{equation}

For each $(t,x)\in \Omega,$ denote
\begin{equation}\label{2.5}\ell (t,x)=\{q\in
\R^n\ | \ \varphi (t,x,q)=\max_{p\in \R^n}\varphi (t,x,p)\}.\end{equation}

\begin{rem}
In virtue of (Hf1), $\ell(t,x)\ne \emptyset,$ for all $(t,x)\in \Omega.$  Moreover, the multi-valued function 
$$\Omega \ni (t,x)\mapsto \ell(t,x)\subset \R^n$$
 is upper semi-continuous, see \cite{vht}.
\end{rem}

\medskip

First, we briefly recall definitions of some
kind of differentials  and viscosity solution as follows.

\begin{defn} Let $u=u(t,x): \ \Omega \to \R$ and let $(t_0,x_0)\in \Omega.$ For $(h,k)\in \R\times \R^n$ we denote 
$$  \tau (p,q,h,k) =\frac{u(t_0+h, x_0+k)-u(t_0,x_0)-ph -\langle q,k\rangle}{\sqrt{|h|^2 +|k|^2}},$$
\end{defn}
$$D^+u(t_0,x_0)=\{ (p,q)\in\R^{n+1}\, |\ \limsup_{(h,k)\to (0,0)} \tau (p,q,h,k) \ \le \ 0\} $$
$$D^-u(t_0,x_0)=\{(p,q)\in \R^{n+1}\, | \ \liminf_{(h,k)\to (0,0)} \tau(p,q,h,k)\ \ge \ 0\},$$
here $p\in \R,\ q\in \R^n.$

Then $D^+u(t_0,x_0)$ (resp. $D^-u(t_0,x_0)$) is called the {\it superdifferential} (resp. {\it subdifferential}) of $u(t,x)$ at $(t_0,x_0).$

\begin{defn}

A continuous functions $u: [0, T)\times \R^n\to \R$ is called a {\it viscosity subsolution} (resp. {\it viscosity supersolution}) of the Cauchy problem (1.1)-(1.2) on $\Omega=(0,T)\times \R^n,$ provided that the following hold:

\smallskip
{\rm (i)} $ u(0,x)=\sigma(x)$ for all $x\in \R^n$;

{\rm (ii)} For each $(t_0,x_0)\in  \Omega$ and $(p,q)\in D^+u(t_0,x_0),$ then
$$p +H(t_0,x_0,q) \le 0,$$
(resp. for each $(t_0,x_0)\in  \Omega$ and $(p,q)\in D^-u(t_0,x_0),$ then
$$p +H(t_0,x_0,q) \ge 0).$$

A continuous function $u: [0,T)\times \R^n \to \R$ is called a viscosity solution of the problem (1.1) -(1.2) if it is a viscosity sub- and supersolution of the problem.

It is noted that, there are several propositions which are equivalent to this definition, e.g., the notion of $C^1$-test function is used instead of semidifferentials, see \cite{cl}.
\end{defn}

 We collect here some properties of Hopf formula $u(t,x)$ and concepts concerning with convexity of a function for further presentation.
 \begin{thm}\label{2.4} Assume (Hf1). Then we have the following:

{\rm 1)} $u(t,x)$ is a convex function on $\Omega$ and it is a Lipschitz solution of the problem (2.1) - (2.2).

{\rm 2)} $u(t,x)$ is a viscosity solution of the problem (2.1) - (2.2).

{\rm 3)} $u(t,x)$ is differentiable at $(t,x)\in \Omega$  and only if for
$\ell(t,x)$ defined by (2.5) is a singleton. Thus, $u(t,x)$ is continuously differentiable in a open set $\mathcal V\subset \Omega$ if $\ell(t,x)$ is a singleton for all $(t,x)\in \mathcal V.$
\end{thm}

\begin{proof} For the proof of 1) see \cite{h,vht}. There are several ways to prove 2), the first proof can be found in \cite{be}. Also see \cite{lr, ai}. 

For the proof of 3), first note that, if $u(t,x)$ is differentiable at $(t_0,x_0)$ then $\ell (t_0,x_0)=\{p\}$  is a singleton, see (\cite{vts}, p. 112). Conversely, if $\ell(t_0,x_0)$ is a singleton, say $\{q\},$ then all partial derivatives of $u(t,x)$ at $(t_0,x_0)$
 exist and $u_x(t_0,x_0)= q,\ u_t(t_0,x_0)=-H(q).$ Since the function $u(t,x)$ is convex, then it is differentiable at this point. Besides, if it is differentiable on $\mathcal V$ then it is continuously differentiable on this open set by a property of convex functions.
\end{proof}

\begin{defn} We call a point $(t_0,x_0)\in \Omega$ {\it regular} for $u(t,x)$ if the function is differentiable at this point. Other point is said to be {\it singular} if at which, $u(t,x)$ is not differentiable.
\end{defn}
Consequently, by Theorem \ref{2.4}, we see that $(t_0,x_0) \in \Omega$ is regular if and only if $\ell(t_0,x_0)$ is a singleton.

\begin{defn}
Let $\mathcal O$ be a convex subset of $\R^m$ and let $v:\ \mathcal O \to \R$ be a continuous function.

(a) The function $v$ is called {\it semiconcave} with linear modulus if there is a constant $C>0$ such that
$$\lambda v(y_1)+(1-\lambda ) v(y_2)-v(\lambda y_1+(1-\lambda)y_2) \le \lambda (1-\lambda)\frac C2 |y_1-y_2|^2$$
for any $y_1,\, y_2$ in $\mathcal O$ and for any $\lambda \in [0,1].$ The number $C$ is called a {\it semiconcavity constant} of $v.$

The function $v$ is called {\it semiconvex} if the function $-v$ is semiconcave.

(b) The function $v$ is called {\it uniformly convex} with constant $\Lambda>0$ if $v(y)-\frac{\Lambda}2|y|^2,\ y\in \mathcal O$ is a convex function.

\end{defn}

\begin{rem}

(i) The theory of semiconcave functions has been fully studied since
the last decades of previous century. The reader is referred to the monograph \cite{cs} for a comprehensive development of the topic.
\smallskip

(ii) The notion of semiconcavity (resp. uniform convexity) is a special case of the notion $\sigma$-smoothness (resp. $\rho$-convexity) of a function, see \cite{ap}. The following proposition is extracted from Prop. 2.6 of the just cited article.
\end{rem}

\begin{prop}\label{scc}Let $v:\R^m \to \R$ be a convex function. Moreover,

(i) Suppose that $v$ is uniformly convex with a constant $C>0.$ Then the Fenchel conjugate function $v^*$ is a semiconcave function with a semiconcavity constant $\frac1C >0.$

(ii) Suppose that $v$ is a semiconcave function with a semiconcavity constant $C^*>0.$ Then $v^*$ is a uniformly convex function with a constant $\frac1{C^*}.$

\end{prop}

\section{A classification of characteristics}

In this section we focus on the study of the relationship between Hopf formula and characteristics. To this aim, let us recall the theory of Cauchy method of characteristics for Problem (\ref{2.1})-(\ref{2.2}).
\medskip

First, by the routine, we assume that $H(p)$ and $\sigma (x)$ are of class $C^2(\R^n).$
\medskip

The characteristic differential equations of Problem
(\ref{2.1})-(\ref{2.2}) is as follows
\begin{equation}\label{2.5}
\dot x=H_p \ ;\qquad \dot v = \ \langle H_p,p\rangle
 - \ H \ ;\qquad \dot p=0 \,\end{equation}
with initial conditions
\begin{equation}\label{2.6} x(0)=y \ ;\qquad v(0)=\sigma(y)\ ;\qquad p(0)=
\sigma _y(y)\ ,\quad y\in \R^n.\end{equation}

Then a characteristic strip of the Cauchy problem
(\ref{2.1})-(\ref{2.2}) (i.e., a solution of the system of
differential equations (\ref{2.5}) - (\ref{2.6})) is defined by
\begin{equation}\label{2.7}\simuleq{x&=x(t,y)=y+t H_p(\sigma_y(y)), \\
v&=v(t,y)=\sigma(y)+t\{
\langle H_p(\tau,\sigma_y(y)),\sigma_y(y)\rangle \}-tH(\sigma_y(y)),\\
 p&= p(t,y)\ =\ \sigma_y(y).}\end{equation}

The first component of solutions (\ref{2.7}) is called a characteristic curve (briefly, characteristics) emanating from $y,$ i.e., the straight line defined by
\begin{equation}\label{2.8}\mathcal C:\ x=x(t,y)=y+t H_p(\sigma_y(y)),\ t\in [0,T]. \end{equation}

Let $t_0\in (0,T].$ If for any $t\in (0,t_0)$ such that $x(t,\cdot):\ \R^n \to \R^n$ is a diffeomorphism, then $u(t,x)=v(t, x^{-1}(t,x))$ is a $C^2$ solution of the problem on the region $(0,t_0)\times \R^n.$

\smallskip
From now on, we make an additional assumption on $H$ and $\sigma.$

\smallskip

(Hf2):  Assume that $H$ and $\sigma$ are functions of class $C^1(\R^n).$ 

\smallskip

Note that, in this case, the characteristic strip (3.3) is also defined.

\smallskip

Let $(t_0,x_0) \in \Omega. $ Denote by $\ell^*(t_0,x_0)$ the set of
all $y\in \R^n$ such that there is a characteristic curve emanating
from $y$ and passing the point $(t_0,x_0).$ We have $\ell (t_0,x_0)
\subset {\sigma}_y(\ell^*(t_0,x_0)),$ see \cite{nh1}. Therefore
$\ell^*(t_0,x_0) \ne \emptyset.$

\begin{prop} Let $(t_0,x_0)\in \Omega.$ Then a characteristic curve passing $(t_0,x_0)$ has form
\begin{equation}\label {2.9} x=x(t,y)=x_0+(t-t_0) H_p(\sigma_y(y)),\ t\in [0,T] \end{equation}
for some $y\in \ell^*(t_0,x_0).$
\end{prop}

\begin{proof} Let $\mathcal C: \ x=x(t,y)=y+t H_p(\sigma_y(y))$ be a characteristic curve passing $(t_0,x_0).$ By definition, $y\in \ell^*(t_0,x_0).$  Then we have

$$x_0=y+{t_0} H_p(\sigma_y(y))$$
Therefore, $$x=x_0-{t_0} H_p(\sigma_y(y))+t H_p(\sigma_y(y)) = x_0+(t-t_0) H_p(\sigma_y(y)).$$

Conversely, let $\mathcal C_1: x=x(t,y)= x_0+(t-t_0) H_p(\sigma_y(y))$ for $y\in \ell^*(t_0,x_0)$ be some curve passing $(t_0,x_0).$ Then  we can rewrite $\mathcal C_1$ as:
\begin{equation}\label{2.10} x=x_0-{t_0} H_p(\sigma_y(y)) +t H_p(\sigma_y(y))=x_0+(t-t_0)H_p(\sigma_y(y)).\end{equation}

On the other hand, let $\mathcal C_2:$
\begin{equation}\label{2.11} x= y +t H_p(\sigma_y(y))\end{equation}
be a characteristic curve also passing $(t_0,x_0).$ Besides that, both $\mathcal C_1,\ \mathcal C_2$ are integral curves of the ODE $x' =H_p(\sigma_y(y)),$ thus they must coincide. This proves the proposition.
\end{proof}

\begin{rem}
Suppose that $\sigma_y(y)=p_0\in \ell(t_0,x_0)$ then $y$ belongs to the subgradient of convex function $\sigma^*$ at $p_0: y\in \partial \sigma^*(p_0).$ Moreover, from (\ref{2.10}) and (\ref{2.11}), we have $y=x_0-t_0H_p(p_0).$
\end{rem}

Now, let $\mathcal C $ be a characteristic curve passing $(t_0,x_0)$ that is written as
$$x=x(t,y)=x_0+(t-t_0) H_p(\sigma_y(y)),\ t\in [0,T]$$

We say that the characteristic curve $\mathcal C$ is of the {\it type (I)  at point} $(t_0,x_0) \in \Omega$, if $\sigma_y(y) =p_0\in \ell(t_0,x_0).$ If $\sigma_y(y)\in \sigma_y(\ell^*(t_0,x_0))\setminus \ell(t_0,x_0)$ then $\mathcal C$ is said of {\it type (II)  at point} $(t_0,x_0).$
\medskip

\smallskip

The following lemma  is helpful in studying Fenchel conjugate of $C^1-$ convex function.
\begin{lem} {\rm {(see} \cite {nh3})}
Let $v$ be a convex function and $D={\rm dom}\; v\subset \R^n.$ Suppose that there exist $p,\,  p_0\in D,\ p\ne p_0$ and $y\in \partial v(p_0)$ such that
$$\langle y,p-p_0\rangle = v(p) -v(p_0).$$

Then for all $z$ in the straight line segment $[p,p_0]$ we have
$$v(z) =\langle y,z\rangle -\langle y,p_0\rangle +v(p_0).$$

Moreover, $y\in \partial v(z)$ for all $z\in [p,p_0].$
\end{lem}
\begin{proof} For the convenience of the reader, we reproduce the proof here.
Take $z=\lambda p +(1-\lambda )p_0\in [p,p_0],\ \lambda \in [0,1].$ Then we have
$$v(z)\le \lambda v(p)+(1-\lambda)v(p_0) =\lambda (v(p)-v(p_0)) +v(p_0).$$

From the hypotheses, we have
$$\aligned v(z)\le &\  \lambda \langle y,p-p_0\rangle +v(p_0)\\
\le &\  \langle y,\lambda p+(1-\lambda) p_0 -p_0\rangle +v(p_0).\endaligned$$

On the other hand, since $y\in \partial v(p_0), $ then
$$\langle y,\lambda p+(1-\lambda) p_0 -p_0\rangle  \le  v(z)-v(p_0).$$

Thus $$v(z)=\langle y,z\rangle -\langle y,p_0\rangle +v(p_0).$$

Next, let $z\in [p,p_0].$ For any $x\in D, $ we have
$$\aligned v(x)-v(z)=&\ v(x)- \langle y,z\rangle +\langle y,p_0\rangle -v(p_0)\\
= &\ v(x)-v(p_0) -\langle y, z-p_0\rangle\\
\ge &\ \langle x-p_0,y\rangle -\langle z-p_0, y\rangle\\
\ge &\  \langle x-z,y\rangle.\endaligned$$

This gives us that $y\in \partial v(z).$
\end{proof}

Now we present  some properties of characteristic curves of type (I) at $(t_0,x_0)$ given by the following theorem.

\medskip

\begin{thm} Assume (Hf1), (Hf2).  Let $(t_0,x_0)\in \Omega = (0,T) \times \R^n,\ p_0= \sigma_y(y_0)\in \ell (t_0,x_0)$ and let
\begin{equation}\label{2.12} \mathcal C: x= x(t)= x_0 +(t-t_0) H_p(p_0), (t,x)\in \Omega \end{equation}
be a characteristic curve of type (I) at $(t_0,x_0).$  Then we have the following:

\smallskip

(i) For all $(t,x)\in \mathcal C, \ 0\le t\le t_0,$ then $p_0\in \ell (t,x).$ Moreover, $\ell(t,x)\subset \ell(t_0,x_0).$ 

\smallskip

(ii) For all $(t,x)\in \mathcal C, \ 0\le t< t_0,$ the set $\ell(t,x)$ is singleton, here $\ell(t,x)=\{p_0\}.$
\end{thm}

As a consequence, if the characteristic curve $\mathcal C: x=x(t)$ is of type (I) at $(t_0,x_0)$ then it is of type (I) at any point $(t_1,x(t_1)), \ t_1\le t_0$ and the Hopf formula is differentiable on a piece of curve $\mathcal C$ corresponding to $t\in [0,t_0).$
\begin{proof}

Take an arbitrary $p\in \R^n$ and denote by
$$ \eta(t,p) =\varphi (t,x,p)-\varphi (t,x,p_0),\ (t,x)\in \mathcal C, \ t\in [0,t_0],$$
where $\varphi (t,x,p)=\langle x,p\rangle -\sigma^*(p)-t H(p).$ Then
\begin{equation}\label{2.13} \eta(t,p)=\langle x(t), p-p_0\rangle -(\sigma^*(p)-\sigma^*(p_0))-t(H(p)-H(p_0))\end{equation}
for $(t,x)\in \mathcal C.$
\smallskip

First, we will check that $\eta(t,p)\le 0$  for all $t\in [0,t_0].$
\smallskip

It is obviously that, $\eta(t_0,p)\le 0.$ On the other hand, from (\ref{2.13}) and Remark 3.2, we have
$$\eta(0,p)=\langle y_0, p-p_0\rangle -(\sigma^*(p)-\sigma^*(p_0)),$$

where $y_0\in \partial \sigma^*(p_0).$ By a property of subgradient of convex function, we have
\begin{equation}\label{2.14}\eta(0,p)=\langle y_0, p-p_0\rangle -(\sigma^*(p)-\sigma^*(p_0))\le  0.\end{equation}

As a result, we have $\eta(0,p) \le  0; \ \eta(t_0,p)\le 0.$
\smallskip

Since $x=x(t)=x_0+(t-t_0)H_p(p_0), $ then from (\ref{2.13}) we also have
$$\eta'(t,p)=\langle H_p(p_0),p-p_0\rangle -(H(p)-H(p_0))=c\  \  ({\text {const}}),\ \forall t\in [0,t_0].$$
\smallskip
Now we begin to prove (i). Fix $(t_1,x_1)\in \mathcal C$ where $ 0\le t_1\le t_0$ and $x_1=x(t_1).$ For any $p\in \R^n, $ we have
\smallskip

+  If $\eta'(t,p)=c> 0$ then $\eta(t_1,p) < \eta(t_0,p)\le 0.$
\smallskip

+ If $\eta' (t,p)=c\le 0$ then $\eta(t_1,p)\le \eta (0,p)\le 0.$
\smallskip

Thus we obtain that for all $p\in \R^n, \varphi (t_1,x_1,p)\le \varphi (t_1,x_1,p_0).$   Consequently, $p_0\in \ell(t_1,x_1)$ for any $(t_1,x_1)\in \mathcal C, \ t_1\in [0,t_0].$
\smallskip

Next, we check that $\ell(t,x)\subset \ell(t_0,x_0), t\in [0,t_0].$ To this end, take $p\in \R^n\setminus \ell(t_0,x_0).$ If $\eta'(t,p)=c\ge 0$ we have
$$\eta (t,p)\le \eta(t_0,p)<0, $$
 and if  $\eta'(t,p)=c< 0,$ then
$$\eta (t,p)< \eta(0,p)=\langle y,p-p_0\rangle -(\sigma^*(p)-\sigma^*(p_0))\le 0, \ t\in [0, t_0).$$
Therefore, in any case, $\eta(t,p)<0.$ This means that $p\notin \ell(t,x)$ and the inclusion $\ell(t,x)\subset \ell(t_0,x_0)$ has been proved.

The proof  of (i) is then complete.
\smallskip

The next step  is to prove (ii). Let $(t_1,x_1)\in \mathcal C$ where $t_1\in [0,t_0).$  Take $p\in \ell(t_1,x_1).$  Then we have
\begin{equation}\label{type}\eta(t_1,p)=\varphi(t_1,x_1,p)-\varphi(t_1,x_1, p_0) =0.\end{equation}

As before, we have $\eta'(t,p)=c\ \ ({\text {const}})$

If $c>0$ then $\eta(t_1, p)<\eta(t_0, p)\le 0$ and if $c<0$ then $\eta(t_1, p)<\eta(0, p)\le 0.$ These yield a contradiction to the equality (\ref{type}).

Now we consider  the case $\eta'(t,p)=0, \ \forall t\in [0, t_0]$ or
\begin{equation}\label{ty} \langle H_p(p_0),p-p_0\rangle -(H(p)-H(p_0))=0.\end{equation}

From the equality (\ref{type}) we have
\begin{equation}\label{typ}\langle x_0, p-p_0\rangle  -(\sigma^*(p)-\sigma^*(p_0))={t_0}(H(p)-H(p_0))\end{equation}

Subtracting both sides of (\ref{typ}) by $\langle {t_0} H_p(p_0), p-p_0\rangle,$ and noticing that $ y_0= x_0-{t_0}H_p(p_0),$ we get
\begin{equation}\label{2.17}\langle y_0, p-p_0\rangle  -(\sigma^*(p)-\sigma^*(p_0))=(H(p)-H(p_0)) -\langle H_p(p_0),p-p_0\rangle\end{equation}
Thus
$$\langle y_0,p-p_0\rangle -(\sigma^*(p)-\sigma^*(p_0))=0.$$

As mentioned before,  since $p_0= \sigma_y(y_0),$ then $y_0\in \partial \sigma^*(p_0).$ 
If $p\ne p_0$ we see that the straight line segment $[p,p_0]$ is contained in $\mathcal D=\{z\in \textrm{dom} \sigma^*\, | \ \partial \sigma^*(z)\ne \emptyset\}.$ Applying Lemma 3.3 we see that the function $\sigma^*$ is not strictly convex on the set $[p,p_0].$ This is a contradiction, since $\sigma(x)$ is of class $C^1(\R^n),$
  then $\sigma^*$ is essentially strictly convex on $D=$ dom$\sigma^*.$ In particular, $\sigma^* $ is strictly convex on $[p,p_0], $ see (\cite{ro}
, Thm. 26.3).
\smallskip
Thus $p=p_0$ and consequently, $\ell (t,x) =\{p_0\}$ for all $(t,x)\in \mathcal C, \ 0\le t < t_0.$
\end{proof}

For a locally Lipschitz function, it is promising to use the notion
of sub- and superdifferential as well as reachable gradients, see
\cite {cs}, e.g.,  to study its differentiability. We use Theorem
3.4 to establish a relationship between $\ell(t_0,x_0)$ and the set
of reachable gradients. 
\smallskip

Let us define the set $D^*u(t_0,x_0)$ of {\it reachable gradients} of a function $u(t,x)$ at $(t_0,x_0)$ as follows:

Given $(p,q)\in \R^{n+1}.$ We say that $(p,q)\in D^*u(t_0,x_0) $ if and only if there exists a sequence $(t_k,x_k)_k\subset \Omega\setminus \{(t_0,x_0)\}$ such that $u(t,x)$ is differentiable at $(t_k,x_k)$ and,
$$(t_k,x_k)\to (t_0,x_0), \ (u_t(t_k,x_k), D_xu(t_k,x_k))\to (p,q)\  \text{ as}\  k\to \infty.$$

If $u(t,x)$ is a locally Lipschitz function,  then $D^*u(t,x) \ne
\emptyset$  and it is a compact set (\cite{cs}, p.54).
\smallskip

Now let $u(t,x)$ be the Hopf formula and let $(t_0,x_0)\in \Omega.$ We denote by
\begin{equation}\label {rg}\mathcal H(t_0,x_0)=\{(-H(q), q)\ | \ q\in \ell(t_0,x_0)\}.\end{equation}
 Then a relationship between $D^*u(t_0,x_0)$ and the set $\ell (t_0,x_0)$ is given by the following theorem.

\begin{thm}
Assume (Hf1), (Hf2). Let $u(t,x)$ be the Hopf formula for Problem (2.1)-(2.2). Then for all $(t_0,x_0)\in \Omega,$ we have
$$D^*u(t_0,x_0)=\mathcal H(t_0,x_0).$$
\end{thm}

\begin{proof}

Let $(p_0,q_0)$ be an element of $\mathcal H(t_0,x_0),$ then $p_0=-H(q_0)$ for some $q_0\in \ell(t_0,x_0).$ Let $\mathcal C$ be the characteristic curve of type (I) at $(t_0,x_0)$ defined as in Theorem 3.4. By this theorem, all points $(t,x)\in \mathcal C, \ t\in [0, t_0)$ are regular. Put $t_k=t_0-1/k, $ then $\mathcal C\ni (t_k,x_k) \to (t_0,x_0)$ and $(u_t(t_k,x_k),D_xu(t_k,x_k))=(-H(q_0),q_0)\to (-H(q_0), q_0)\in D^*u(t_0,x_0)$ as $k\to \infty.$  Therefore, $\mathcal H(t_0,x_0)\subset D^*u(t_0,x_0).$

On the other hand, let $(p,q)\in D^*u(t_0,x_0)$ and $(t_k,x_k)_k\subset \Omega\setminus \{(t_0,x_0)\}$ such that $u(t,x)$ is differentiable at $(t_k,x_k)$ and,
$$(t_k,x_k)\to (t,x), \ (u_t(t_k,x_k), D_xu(t_k,x_k))\to (p,q)\  \text{ as}\  k\to \infty.$$
Since \qquad $ (u_t(t_k,x_k), D_xu(t_k,x_k))= (-H(q_k),q_k) \ { \text{for}}\  q_k\in \ell(t_k,x_k),$

\noindent and multivalued function $\ell(t,x)$ is u.s.c, then letting $k\to \infty,$ we see that $q\in \ell(t_0,x_0)$ and $p=\lim_{k\to \infty} -H(q_k) =-H(q).$ Thus $(p,q)\in \mathcal H(t_0,x_0).$ The theorem is then proved.
\end{proof}

\begin{rem} A general result for the correspondence between
$D^*u(t,x)$ and the set of minimizers of $(CV)_{t,x}$ is established
for convex Hamiltonian $H(t,x,p)$ in $p$ in \cite{cs}, Th. 6.4.9,
p.167.
\end{rem}
\section{ Existence of a strip of differentiability of Hopf formula}

First, we present the following result on the existence of  strips of the form $\mathcal R=(0,t_*)\times \R^n\subset \Omega$ such that on which the viscosity solution $u(t,x)$ defined by Hopf formula is continuously differentiable. 

\begin{thm}\label{sc} Assume  (Hf1). Suppose that the Hamiltonian  $H=H(p)$ is a semiconvex function with the semiconvexity constant $\gamma >0.$  In addition, let $\sigma$ be a semiconcave function with semiconcavity constant $\mu^{-1}>0.$ Then there exists $t_*\in (0, T)$ such that for all $t_0\in (0, t_*),$ the function $v(x)=u(t_0,x)$ is semiconcave, where $u(t,x)$ is the Hopf formula  defined by (\ref{2.3}).
 \end{thm}

\begin{proof} We follow an argument in the proof of Theorem 3.5.3 (iv) \cite{cs} with an appropriate adjustment. 

By assumption and Prop. \ref{scc}, we first note that the Fenchel conjugate function $\sigma^*$ is a uniformly convex function with constant $\mu>0.$ Therefore the function $\sigma^*(p) -\frac{\mu}2 |p|^2$ is convex and then, for all $a,b\in \R^n$  we have
\begin{equation}\label{sc1} \sigma^*(a) +\sigma^*(b) -2\sigma^*(\frac{a+b}{2}) \ge \frac{\mu}{2} (|a|^2+|b|^2-2|\frac{a+b}{2}|^2)=\frac{\mu}4|a-b|^2.
\end{equation}

 Now, take  $t_*\in (0,T)$ such that $0<\gamma t_*\le \frac \mu 2.$ Let $ t_0\in (0,t_*), x, y\in \R^n,$ pick out $p\in \ell(t_0,x), q\in \ell(t_0,y);$ using the inequality (\ref{sc1}) we have
$$\aligned  u(t_0,x)&+u(t_0,y)-2u(t_0,\frac{x+y}2)\le \langle x,p \rangle -\sigma^*(p)  -t_0H(p) +\langle y,q\rangle -\sigma^*(q)\\
&-t_0H(q)-2\Big( \langle \frac {x+y}2,\frac{p+q}2\rangle -\sigma^*(\frac{p+q}2)-t_0H(\frac{p+q}2)\Big)\\
\le2 \Big( \sigma^*(\frac{p+q}2)&-\frac{\sigma^*(p)+\sigma^*(q)}2\Big)+\frac 12 \langle x-y,p-q\rangle +2t_0(H(\frac{p+q}2)-\frac{H(p)+H(q)}2 \big)\\
&\le -\frac {\mu} 4 |p-q|^2 +\frac 14(2\langle x-y, p-q\rangle)+2t_0(\frac \gamma 8|p-q|^2\\
&\le -\frac {\mu} 4 |p-q|^2 +t_0(\frac \gamma 4|p-q|^2+\frac 14(\frac \mu 2|p-q|^2+\frac 2 \mu|x-y|^2)\\
&\le \frac 14(t_0\gamma -\frac{ \mu} 2) |p-q|^2 +\frac 1{2\mu} |x-y|^2.\endaligned $$

(Above, we use an obvious inequality of the form $2\langle x-y,p-q\rangle \le \frac \mu 2|p-q|^2 +\frac 2 \mu |x-y|^2.)$ 

Therefore,
$$ u(t_0,x) +u(t_0,y)-2u(t_0,\frac{x+y}2)\le  \frac 1{ 2\mu} |x-y|^2.$$

Thus, the function $v(x)=u(t_0,x)$ is a semiconcave function.
\end{proof}

\begin{cor} Suppose that all assumptions of Theorem \ref{sc1} hold. Then $u(t,x)$ defined by Hopf formula is of class $C^1(0,t_*)\times\R^n,$ where $0<\gamma t_*\le \frac {\mu}2.$
\end{cor}

\begin{proof} 
Let $(t_0,x_0)\in (0,t_*)\times \R^n. $ By Theorem \ref{sc1}, the function $v(x)=u(t_0,x)$ is  a semiconcave on $\R^n.$ Moreover, $v(x)$ is also a convex function. By Theorem 3.3.7 \cite{cs}, the function $v(x)=u(t_0,x)$ is of class $C^1(\R^n).$ Thus, $\ell(t_0,x_0)$ is a singleton and then $u(t,x)$ as function of two variables, is differentiable at $(t_0,x_0).$ Moreover, $u(t,x)$ is a convex function, therefore $u(t,x)$ is of class $C^1((0,t_*)\times\R^n).$
\end{proof}

Being inspired by Lemma 6.5.1 \cite{cs} we can derive the following lemma which is useful in studying the differentiability of Hopf formula.

\begin{lem} Assume (Hf1), (Hf2). In addition, suppose that $\sigma(x)$ is Lipschitz on $\R^n.$ Let $(t_0,x_0)\in [0,T).$  Moreover, suppose  that there exist $t_*\in (0,T)$ such that $\ell(t_*,y)=\{p(y)\}$ is a singleton, for all $y\in \R^n.$ Then there exists $x_*\in \R^n$ and a characteristic curve $\mathcal C$ of type (I)  at $(t_*,x_*):  x=x_*+(t-t_*)H_p(p(x_*))$ passing $(t_0,x_0).$
\end{lem}

\begin{proof} 
Following Remark 2.1, the multi-valued function    $y\mapsto \ell(t_*,y)$  is upper semi-continuous. By assumption, $\ell(t_*,y)=\{p(y)\},$ thus the single-valued function $y\mapsto p(y)$ is continuous on $\R^n.$

For all $y\in \R^n,$ let 
$$\Lambda (y)=x_0-(t-t_*) H_p(p(y)),$$ then the function $\Lambda$ is also continuous on $\R^n.$

Since $\sigma(x)$ is convex and Lipschitz, then $D=$\ dom$\sigma^*$ is bounded. Hence, $D\subset B'(0,M)$  for some positive number $M.$ Let $N=(t_*-t_0) \sup_{|p|\le M}\, |H_p(p)|.$
\smallskip

Note that, if $y\in B'(x_0,N)$ then
$$|\Lambda (y)-x_0|\le (t^*-t_0) |H_p(p(y)| \le N.$$
Therefore $\Lambda$ is a continuous function from the closed ball $B'(x_0,N)$ into itself. By Brouwer theorem, $\Lambda $ has a fixed point $x_*\in B'(x_0,N), $ i.e., $\Lambda (x_*)=x_*,$ hence,
$$x_0=x_*+(t_0-t^*)H_p(p(x_*)).$$

In other words, there exists a characteristic curve $\mathcal C$ of the type (I) at $(t_*,x_*)$ described as in Prop. 3.1  passing $(t_0,x_0)$. The lemma is then proved.
\end{proof}

\begin{rem}
By Cauchy method of characteristics and by assumptions that $H$ and $\sigma$ are of class $C^2(\R^n),$  the unique $C^2$-solution $u(t,x)$ of problem (2.1) - (2.2) exists in a narrow neighborhood of the hyperplane $t=0$ where characteristic curves do not meet. Nevertheless, if $u(t,x)$ given by Hopf fomula is differentiable in some open set containing $(t_0,x_0)\in \Omega,$ then several characteristic curves may cross at $(t_0,x_0)$ as in the following example.

Consider the following problem
$$u_ t-\Big(1+|u_x|^2\Big)^{\frac 12} =0,\
t>0,\ x\in \R ,$$
$$u(0,x)=\frac{x^2}{2},\ x\in \R.$$
The Hopf formula of this problem is:
$$u(t,x)=\max_{y\in \R} \{xy-\frac{y^2}{2}+t (1+y^2)^{\frac
12}\}.$$ 

By computing, we recognize that $\ell(t,x)$ is a singleton for all points in the region $\mathcal R^*=((0,+\infty )\times  \R)\setminus \{(t,0),\ t\ \ge 1\}.$
Thus, the solution $u(t,x)$ is continuously differentiable in this region. Using method of characteristics, we see that when $t>1,$
the  characteristic curves intersect. Concretely,
two  curves of the form $x(t,y)=y-\displaystyle\frac{ty}{\sqrt{1+y^2}}$  starting
from $y_0 = 1$ and $y_1 =2$ 
meet each other at  the point $\big(\frac{\sqrt{10}}{2\sqrt 2-\sqrt 5},\
\frac{2(\sqrt 
2-\sqrt 5)}{2\sqrt 2-\sqrt 5}\big )\in \mathcal R^*,$ but the differentiability  of the solution $u(t,x)$ is also preserved in some neighbourhood of this point.

\smallskip
However, if Hopf formula $u(t,x)$ is differentiable on a whole strip of the form $(0,t_0)\times \R^n$ then the situation is different. More specific, we have the following theorem as a necessary condition.
\end{rem}

\begin{thm}\label{cross}
Assume (Hf1), (Hf2). Suppose that $u(t,x)$ is differentiable on a strip $\mathcal R=(0,T_0)\times \R^n, \ T_0<T.$ Then at any point $(t_0,x_0)\in \mathcal R$ there are no characteristic curves crossing each other. 
\end{thm}

\begin{proof}
On the contrary, suppose that there are two distinct characteristic curves  $\mathcal C_i: x=y_i+tH_p(\sigma_y(y_i)), \ i=1,2,\ y_1\ne y_2$ meet at  $(t_0,x_0).$ If both $\mathcal C_i, i=1,2$ are of type (I) at $(t_0,x_0)$, then $\{p_1, p_2\}\subset \ell(t_0,x_0),$ where $p_1=\sigma_y(y_1)\ne \sigma_y(y_2)=p_2.$ This means that $u(t,x)$ is not differentiable at this point. Therefore, at least a $\mathcal C_i, i=1,2,$ say, $\mathcal C_1$ is of type (II) at $(t_0,x_0).$ Let
$$t_+=\inf \{t\in [0,t_0)\ |\ \mathcal C_1 \ {\text{is of type (II) at} }\  (t, x_1(t)\}.$$
Consider the point $(t_+, x_+)$ where $x_+=x_1(t_+)$ then $p(x_1)=p(x_+).$ Take $t_*\in (t_0, T_0).$ By assumption, $\ell(t_*, x)$ is a singleton for all $x\in \R^n.$ Applying Lemma 4.3, there exists a point $(t_*, x_*)\in\mathcal R$ and a characteristic curve $\mathcal C': x=x_*+( t-t_*)H_p(p(x_*))$ of type (I) at $(t_*,x_*)$ and passing $(t_0,x_0).$ 

+ If $p(x_1) =p(x_*)$ then $\mathcal C_1=\mathcal C'.$ By Theorem 3.4, the characteristic curve $\mathcal C_1$ is of type (I) at $(t_0,x_0).$This is a contradiction. 

+ If $p(x_1)\ne p(x_*)$  (i.e. $\mathcal C_1\ne \mathcal C',$  and $0<t_+\le t_0)$ then $\mathcal C_1$ is of type (I) at all points $(t,x_1(t)),\ 0\le t<t_+.$ Thus, $\{(-H(p(x_+)),p(x_+)), (-H(p(x_*)), p(x_*))\}\subset D^*u(t_+, x_+).$ It follows that $u(t,x)$ is not differentiable at $(t_+, x_+).$ This also contradicts to the hypothesis of the theorem.

The proof Theorem \ref{cross} is now complete.
\end{proof}

Next, we present some sufficient conditions so that there exists a strip of the form $(0,t_*)\times \R^n$ on which the function $u(t,x)$ is differentiable. The first result is concerned with non-crossing characteristics conditions, i.e. $\ell^*(t_*,x)$ is a singleton.

\begin{thm} Assume (Hf1), (Hf2). Let $u(t,x)$ be the viscosity solution of Problem (\ref{2.1}) - (\ref{2.2}) defined by Hopf formula (2.3). Suppose that there exists $t_* \in (0,T)$ such that the mapping: $y\mapsto x(t_*,y)=y+{t_*} H_p(\sigma_y (y))$ is injective. Then $u(t,x)$ is continuously differentiable in the open strip $(0,t_*)\times \R^n.$
\end{thm}

\begin{proof}

Let $(t_0,x_0)\in (0,t_*)\times \R^n$ and let $\mathcal C:$
$$ x=x_0+(t-t_0) H_p(p_0)$$
where $p_0=\sigma_y(y_0)\in \ell(t_0,x_0),$  be the characteristic curve going through $(t_0,x_0)$ defined as in Proposition 3.1.

Let $(t_*,x_*)$ be the intersection point of $\mathcal C$ and plane $P^{t_*} :\ t=t_*.$ By assumption, the mapping $y\mapsto x(t_*,y)$ is injective and $\ell(t_*,x_*)\ne \emptyset,$ so there is a unique characteristic curve passing $(t_*,x_*).$ This characteristic curve is exactly $\mathcal C.$ Therefore, we can rewrite $\mathcal  C$ as follows:
$$ x=x_*+(t-t^*)H_p(p_*)$$
where $p^*\in \ell(t_*,x_*).$

Since $\ell^*(t_*,x_*)$ is a singleton, so is $\ell(t_*,x_*).$
Consequently, $\mathcal  C$ is of type (I) at $(t^*,x^*)$ and $\ell(t,x)
=\{p^*\}$ for all $(t,x)\in \mathcal  C,$  particularly at $(t_0,x_0)$ and
then, $p^*=p_0.$ Applying Theorem 2.4 we see that $u(t,x)$ is of
class $C^1((0,t_*)\times \R^n).$
\end{proof}

The next theorem concerns with the single-valuedness of the set of maximizers $\ell(t_*,x).$

\begin{thm} Assume (Hf1), (Hf2). In addition, suppose that $\sigma(x)$ is Lipschitz on $\R^n.$ If $\ell(t_*,x)$ is a singleton for every point of the plane $P^{t_*} = \{(t_*,x)\in \R^{n+1}:\  x\in \R^n\},$ for some $t_*\in (0,T),$ then the function $u(t,x)$ defined
by Hopf formula (\ref{2.3}) is continuously differentiable in the open strip
 $(0,t_*) \times \R^n.$\end{thm}

\begin{proof} Let $(t_0,x_0) \in (0,t_*)\times \R^n.$
By Lemma 4.3 there exists a characteristic curve $\mathcal C$ of the type (I) at $(t_*,x_*)$ passing $(t_0,x_0)$. Since $\ell(t_*,x_*)$ is a singleton, so is $\ell(t_0,x_0).$ Applying Theorem 2.4, we see that $u(t,x)$ is continuously differentiable in $(0,t_*)\times\R^n.$
\end{proof}

We note that the hypotheses of above theorems are equivalent to the
fact that, there is unique characteristic curve of type (I) at a regular
point $(t_*,x),\ x\in \R^n$ going through the point $(t_0,x_0),$ thus this point is also regular. In general,
at some point $(t_0,x_0) \in (0,t_*)\times \R^n$ where $u(t,x)$ is
differentiable there may be more than one characteristic curves of
type (I) or (II) at such a point $(t_*,x),\ x\in \R^n,$ passing, that is
$\ell^*(t_*,x)$ need not be a singleton. Even neither is
$\ell(t_*,x),$ see Remark 4.4. Nevertheless, we have:

\begin{thm} Assume (Hf1), (Hf2). Let $u(t,x)$ be the viscosity solution of Problem (\ref{2.1}) - (\ref{2.2}) defined by Hopf-type formula. Suppose that there exists $t_* \in (0,T)$ such that all characteristic curves passing $(t_*,x),\ x\in \R^n$ are of type (I). Then $u(t,x)$ is continuously differentiable in the open strip $(0,t_*)\times \R^n.$
\end{thm}

\begin{proof}

We argue similarly to the proof of Theorem 4.1. Let $(t_0,x_0)\in (0,t_*)\times \R^n$ and let $\mathcal C:$
$$ x=x_0+(t-t_0)H_p(p_0)$$
where $p_0=\sigma_y(y_0)\in \ell(t_0,x_0)$  be the characteristic curve going through $(t_0,x_0)$ defined as in Proposition 3.1.

Let $(t_*,x_*)$ be the intersection point of $\mathcal C$ and plane $P^{t_*} :\ t=t_*.$  Then we have
$$ x_*=x_0+(t^*-t_0) H_p(p_0)$$
Therefore, we can rewrite $\mathcal C$ as
$$ x=x_*-(t^*-t_0)H_p(p_0)  + (t-t_0) H_p(p_0) = x_*+(t-t^*) H_p(p_0),$$
then it is also a characteristic curve passing$(t_*,x_*).$ By assumption, $\mathcal C$ is of type (I) at this point, so all $(t,x)\in \mathcal C, \ 0\le t <t_*$ are regular by Theorem 3.4. Thus, $\ell(t_0,x_0)$ is a singleton. As before, we come to the conclusion of the theorem.
\end{proof}

\section{Propagation of Singularities of Hopf formula}

In the previous section we see that, under some conditions, the Hopf formula is continuously differentiable on a strip of the form $(0,t_0))\times \R^n.$ Let
$$\theta =\sup\{t\in (0,T)\ |\ u(t,x)\in C^1((0,t)\times\R^n)\}.$$

Then $\mathcal R_\theta =(0,\theta)\times \R^n$ is the largest strip on which $u(t,x)$ is continuously differentiable. Therefore, for any $t_0>\theta,$ there exists $x_0\in \R^n$ such that $u(t,x)$ is not differentiable at the point $(t_0,x_0),$ i.e. $(t_0,x_0)$ is a singular point of $u(t,x).$

\smallskip

Next, we study the propagation of singularities of viscosity solution $u(t,x)$ of the Cauchy problem (2.1)-(2.2). The first theorem presents a simple propagation of singularities from a point $(t_0,x_0)$ to the boundary as $t=T$  on a ``cone'' with vertex $(t_0,x_0)$ and base $B'(x_0, L)$ for some $L>0.$ The second one establishes the propagation of singularities on some Lipschitz curve.

\medskip

We start with the following lemma.

\begin{lem}
Assume (Hf1), (Hf2). Moreover, let $\sigma (x)$  be a Lipschitz function on $\R^n.$ Then for each $\epsilon >0,$ there exists $\delta >0$ such that if $(t_0,x_0)$ is a singular point for $u(t,x),$ then
for any $t_1\in [t_0,t_0+\delta]$ there exists $x_1\in B'(x_0,\epsilon)$ such that $(t_1,x_1)$ is also a singular point.
\end{lem}

\begin{proof}
Since $\sigma(x)$ is convex and Lipschitz, then $D=$\ dom$\sigma^*$ is bounded. Hence, $D\subset B'(0,M)$  for some positive number $M.$ Let $\epsilon >0.$ We choose a fixed number $\delta >0$ such that $\delta\, \sup_{|p|\le M}\, |H_p(p)| \le \epsilon.$
\smallskip

Let $(t_0,x_0)$ be a singular point. Suppose contrarily that there exists $t_*\in [t_0, t_0+\delta]$ such that for any $y\in B'(x_0, \epsilon),$ the point $(t_*, y)$ is regular. Then $\ell(t_*,y) = \{p(y)\} =\{p(t_*,y)\}, \ y\in B'(x_0, \epsilon)$ is a singleton. We argue similarly as in the proof of Lemma 4.3. Since the multi-valued function  $y\mapsto \ell(t_*,y)$ is u.s.c, then $y\mapsto p(y)$ is continuous on $B'(x_0,\epsilon).$ Therefore the function $B'(x_0,\epsilon) \ni y\mapsto \Lambda (y)=x_0-(t_0-t_*) H_p(p(y))$ is also continuous.
\medskip

Note that, if $y\in B'(x_0,\epsilon)$ then
$$|\Lambda (y)-x_0|\le (t^*-t_0) |H_p(p(y)| \le \delta\, \sup_{|p|\le M}\, |H_p(p)| \le \epsilon.$$
Therefore $\Lambda$ is a continuous function from the closed ball $B'(x_0,\epsilon)$ into itself. By Brouwer theorem, $\Lambda $ has a fixed point $x_*\in B'(x_0,\epsilon), $ i.e., $\Lambda (x_*)=x_*,$ hence,
$$x_0=x_*+(t_0-t_*)H_p(p(x_*)).$$

In other words, there exists a characteristic curve $\mathcal C$ of the type (I) at $(t_*,x_*)$ described as in Theorem 3.4 passing $(t_0,x_0)$. Since $\ell(t_*,x_*)$ is a singleton, so is $\ell(t_0,x_0)$  and therefore $(t_0,x_0)$ is regular. This yields a contradiction.
\end{proof}

\begin{thm}
Assume (Hf1), (Hf2). Moreover, let $\sigma (x)$  be a Lipschitz function on $\R^n.$ Let $(t_0,x_0)\in \Omega$  be a singular point of Hopf formula $u(t,x)$ and $\epsilon >0.$ Then for any $t'\in (t_0,T)$ there exists $x'\in\R^n$ such that $(t',x')$ is a singular point for $u(t,x)$ and $|x'-x_0|\le (m+1)\epsilon$ for some $m\in \N.$
\end{thm}

\begin{proof}
Let $\epsilon >0.$ Given $t'\in (t_0, T)$ there is $m\in \N$ such that $t_0+m\delta <t'\le t_0+(m+1)\delta,$  where $\delta $ is defined by Lemma 5.1. By induction,  after $m$ steps, let $(t_0+m\delta, x_m), \ x_m\in B'(x_{m-1},\epsilon)$ be the singular point  $(t_0,x_0)$ in Lemma 5.1,  we can take $x'\in B'(x_m,\epsilon)$ such that $(t',x')$ is a singular point. It is easily seen that
$$|x'-x_0|\le |x'-x_m| +|x_m-x_{m-1}| +\dots +|x_1-x_0| \le (m+1)\epsilon.$$
The theorem is now proved.
\end{proof}
Next, we show that singularities of Hopf  formula $u(t,x)$ can propagate on some Lipschitz arc.
\begin{thm}
Assume (Hf1), (Hf2). Moreover, let $\sigma (x)$  be a Lipschitz function on $\R^n.$ Let $(t_0,x_0)\in \Omega$  be a singular point of Hopf formula $u(t,x).$Then there exist a positive number $\rho$ and a Lipschitz arc $\mathcal S: [0,\rho] \ni s\mapsto x=x(s) \in \R^{n+1}$ with $x(0)=(t_0,x_0), \ x(s)\ne x(0)$ when $s\ne 0$ and $(s,x(s))$ is a singular point for $u(t,x), s\in [0, \rho].$
\end{thm}

\begin{proof}
Since $u(t,x)$ is not differentiable at $(t_0,x_0)$ then $\ell(t_0,x_0)$ contains more than one element, say $q_1, q_2\in \ell(t_0,x_0).$ By Theorem 3.5 one has 
$(p_i, q_i)\in D^*u(t_0,x_0)\subset D^-u(t_0,x_0), $ where $p_i=-H(q_i), \ i=1,2.$

Let $g(p,q)=p+H(q), \ (p,q)\in \R\times \R^n.$ Then $g$ is continuous on $\R^{n+1}.$ On the other hand, from $D^-u(t_0,x_0)={\rm co}D^*u(t_0,x_0)$ is a compact set and by definition of viscosity solution, $g(p,q)=p+H(q)\ge 0, \forall (p,q)\in D^-u(t_0,x_0), $ we deduce that
$$\max_{(p,q)\in D^-u(t_0,x_0)} g(p,q) =p_0+H(q_0) =\alpha \ge 0,\ (p_0,q_0)\in D^-u(t_0,x_0).$$

If $\alpha >0,$ then for any $n\in \N, $
$$p_0+\frac 1n +H(q_0)=\alpha +\frac 1n >\alpha, $$
thus $(p_0+\frac 1n, q_0)\notin D^- u(t_0,x_0).$ Therefore, $(p_0, q_0)\in \partial D^-u(t_0,x_0)\setminus D^*u(t_0,x_0)\ne \emptyset.$ By Theorem 4.2.2 \cite{cs}, we get the desire result.

\smallskip

It remains to show that $\alpha >0.$ Suppose contrarily, then $g(p,q)=0, $ for all $(p,q)$ in the straight line segment $[(p_1,q_1), (p_2,q_2)],$  i.e.
$$\forall \lambda \in [0,1], \ \lambda p_1+(1-\lambda )p_2 +H(\lambda q_1+(1-\lambda q_2)) =0.$$

On the other hand, since $(p_i, q_i)\in D^*u(t_0,x_0), i=1,2$ we have
$$p_1+H(q_1)=0; \quad p_2+H(q_2)=0.$$
Thus, $$\lambda p_1+(1-\lambda )p_2 +\lambda H(q_1) +(1-\lambda) H(q_2) =0.$$
Therefore, 
$\lambda H(q_1) +(1-\lambda) H(q_2) = H(\lambda q_1+(1-\lambda q_2)).$

Since $q_i, i=1,2 \in \ell(t_0,x_0),$ we have
\begin{equation}\label{sing} \langle x_0, q_1\rangle -\sigma^*(q_1)-t_0H(q_1)= \langle x_0, q_2\rangle -\sigma^*(q_2)-t_0H(q_2)= u(t_0,x_0)=\theta.\end{equation}
Moreover, 
$$\langle x_0, \lambda q_1+(1-\lambda )q_2\rangle -\sigma^*(\lambda q_1+(1-\lambda) q_2)-t_0H(\lambda q_1+(1-\lambda) q_2)\le \theta.$$
Combining this and equalities (\ref{sing}), we get
$$\lambda (\theta +\sigma^*(q_1)) +(1-\lambda) (\theta +\sigma^*(q_2))- \sigma^*(\lambda q_1+(1-\lambda)q_2)\le \theta.$$
Thus $$\lambda (\sigma^*(q_1)) +(1-\lambda) (\sigma^*(q_2))\le  \sigma^*(\lambda q_1+(1-\lambda)q_2)$$ and therefore
$$\lambda (\sigma^*(q_1)) +(1-\lambda) (\sigma^*(q_2))=\sigma^*(\lambda q_1+(1-\lambda)q_2),$$ since $\sigma$ is convex. 

Note that $[q_1,q_2]$ is contained in $\mathcal D=\{z\in \textrm{dom} \sigma^*\, | \ \partial \sigma^*(z)\ne \emptyset\}.$ Since $q_1\ne q_2,$ we apply Lemma 3.3 to see that the function $\sigma^*$ is not strictly convex on the straight line segment $[q_1,q_2].$ On the other hand, by assumption, the convex function $\sigma(x)$ is of class $C^1(\R^n),$  then its conjugate function $\sigma^*$ is essentially strictly convex on $D=$ dom$\sigma^*.$ In particular, $\sigma^* $ is strictly convex on $[q_1,q_2], $ see (\cite{ro}, Thm. 26.3). This is a contradiction.
\end{proof}

\noindent{\bf Example.} Let

$$u_t-\ln(1+u_x^2) =0,\
t>0,\ x\in \R ,$$
$$u(0,x)=\begin{cases} \frac {x^2}2, & |x|\le 1\\
x{\rm sgn} x-\frac 12, & |x|>1\end{cases}$$

A  viscosity solution defined by Hopf formula of this problem is:
$$u(t,x)=\max_{|y|\le 1}\, \{xy-\frac{y^2}{2}+t \ln (1+y^2)\}$$
Let $\varphi (t,x,y)=xy -\frac {y^2}{2} +t\ln (1+y^2),$ then $\varphi_y(t,x,y)=x-y+\frac{2ty}{1+y^2}.$

A simple computation shows that at point $(t_0,x_0)= (2, \frac 25),$ we have $ \varphi_y(2,\frac 25,y)=0\  \Leftrightarrow\  y_1=2; \ y_2=\frac{ -4+ \sqrt {11}}5,\  y_3=\frac{ -4- \sqrt {11}}5$ and the function $\varphi (t_0,x_0,y)$ attains its maximum at $y_1=2.$
\smallskip

There are three characteristic curves that go through the point $(2, \frac 25)$ as follows:

$\mathcal C_1: \ x=2-\frac {4t} 5,\quad  \text{starting at y=2}$ and

$\mathcal C_{i} = y_{i} -\frac{2y_{i}t}{1+y_{i}^2},\ i=2,\, 3, \  \text{starting at}\ y_{2}= \frac{ -4+ \sqrt {11}}5,\ y_{3}= \frac{ -4- \sqrt {11}}5 .$

We see that $\mathcal C_1$ is the characteristic curve of type (I) at $(t_0,x_0)$ and $\mathcal C_{2}, \ C_{3}$ are the characteristic curves of type (II) at $(2, \frac 25)$ since $\ell (2, \frac 25)=\{\sigma'(y_1)\} =\{2\}$ and $\sigma_y(y_i)\notin \ell(2, \frac 25),\ i=2,\, 3.$ Note that, $(2,\frac 25)$ is a regular point of $u(t,x).$

\medskip

Now let $(t_1, x_1)=(t_1,0)$ and let the characteristics $\mathcal C'_1$ starting from $y\in \R$ goes through $(t_1,0).$ Then $y$ is a root of equation
$y-\frac{2t_1^2y}{1+y^2}=0.$

If $0\le t_1\le \frac 1{2}$ then $(t_1,0)$ is regular point of $u(t,x)$ and $\mathcal C'_1:\ x=0$ is of type (I) at $(t_1,0).$

If  $t_1> \frac 1{2}$ then $(t_1,0)$ is singular, since $\ell(t_1,0) =\{y_2,\, y_3\},$ where $y_2=\sqrt{2t_1 -1},\ y_3=-\sqrt{2t_1-1}.$ In this case, the characteristic curves $\mathcal C'_2$ and $\mathcal C'_3$ starting at $y_2$ and $y_3$ are of type (I), and $\mathcal C'_1$ is of type (II) at $(t_1, 0).$

Let $t_* = \frac 1{ 2}.$ We have  $\varphi (\frac 1{ 2},x,y)=xy -\frac {y^2}{2} +\frac 12\ln (1+y^2),$ then $\varphi'_y(\frac 1{2},x,y)=x-y+\frac{y}{1+y^2}$ and $\varphi''_y(\frac 1{\sqrt 2},x,y)=-y^2\frac{3+y^2}{(1+y^2)^2} <0, \ y\ne 0.$ Therefore $\ell(  \frac 1{2},x)$ is a singleton for all $x\in \R.$ Applying Theorem 4.7, we see that the solution $u(t,x)$ is continuously differentiable on the strip $(0, \frac 1{\sqrt 2})\times \R^n.$
\smallskip

At last, the segment $x=0;\ t\in (\frac 1{2},T]$ is a set of singular points for $u(t,x).$ So the singularities of $u(t,x)$ propagate to the boundary.
\bigskip

{\bf Acknowledgments.} This research is funded by Vietnam National Foundation for Science and Technology Development (NAFOSTED) under grant number:  101.02-2013.09. A part of this paper was done when the author was working at the Vietnam Institute for Advance Study in Mathematics (VIASM). He would like to thank the VIASM for financial support and hospitality.

\bigskip


\end{document}